\newif\ifarxiv
\arxivtrue

\documentclass[letterpaper, 10 pt, conference]{ieeeconf}
\IEEEoverridecommandlockouts                              

\usepackage{fix-cm}
\usepackage{etex}

\usepackage{dblfloatfix}

\usepackage{nag}

\makeatletter
\@ifpackageloaded{xcolor}{}{%
\usepackage[table,x11names,dvipsnames,svgnames]{xcolor}%
}
\makeatother

\usepackage{colortbl}

\usepackage{graphicx}
\usepackage{wrapfig}

\definecolorset{RGB}{lyft}{}{Red,194,39,36;Sunset,202,53,33;Orange,205,68,20;Amber,200,117,42;Yellow,242,169,52;Citron,186,188,44;Lime,112,159,33;Green,56,139,31;Mint,45,118,56;Teal,52,133,135;Cyan,60,132,202;Blue,55,94,248;Indigo,64,13,247;Purple,115,42,248;Pink,176,25,145;Rose,176,32,75}

\usepackage{cite}

\usepackage{microtype}

\usepackage[american]{babel}

\usepackage{array}
\usepackage{multirow}
\usepackage{booktabs}
\usepackage{makecell} 

\ifcsname labelindent\endcsname
\let\labelindent\relax
\fi
\usepackage[inline]{enumitem}

\usepackage{subfig}

\newenvironment{lenumerate}[2][]
{\begin{enumerate}[label=(#2\arabic*),leftmargin=0.2in,itemindent=0.15in,#1]}
{\end{enumerate}}

\setlist*[enumerate,1]{label={\itshape\arabic*)}}

\makeatletter
\newcommand{\paragraphswithstop}{%
\let\copyparagraph\paragraph%
\renewcommand\paragraph[1]{\copyparagraph{##1.}}%
}
\makeatother

\usepackage[framemethod=tikz]{mdframed}

\makeatletter
\def\namedlabel#1#2{\begingroup
  #2%
  \def\@currentlabel{#2}%
  \phantomsection\label{#1}\endgroup
}
\makeatother

\usepackage{suffix}

\usepackage{environ}

\makeatletter
\newsavebox{\boxifnotempty}
\newcommand{\displayifnotempty}[3]{\sbox\boxifnotempty{#2}\setbox0=\hbox{\usebox{\boxifnotempty}\unskip}%
\ifdim\wd0=0pt
\else
 #1\usebox{\boxifnotempty}#3%
\fi%
}

\newcommand{\ifempty}[2]{\setbox0=\hbox{#1\unskip}%
\ifdim\wd0=0pt%
 #2%
\fi%
}

\newcommand{\ifnotempty}[2]{\setbox0=\hbox{#1\unskip}%
\ifdim\wd0>0pt%
 #2%
\fi%
}
\makeatother

\newcommand{\switchifempty}[3]{\sbox\boxifnotempty{#1}\setbox0=\hbox{\usebox{\boxifnotempty}\unskip}%
  \ifdim\wd0=0pt{}%
  #2%
  \else{}%
  #3%
  \usebox{\boxifnotempty}%
  \fi%
}

\usepackage{algorithm}

\usepackage{scrlfile}

\makeatletter
\newcommand*\newstoreddef[1]{
  \BeforeClosingMainAux{%
    \immediate\write\@auxout{%
      \string\restoredef{#1}{\csname #1\endcsname}%
    }%
  }%
}
\newcommand*{\restoredef}[2]{
  \expandafter\gdef\csname stored@#1\endcsname{#2}%
}
\newcommand*{\storeddef}[1]{
  \@ifundefined{stored@#1}{0}{\csname stored@#1\endcsname}%
}
\makeatother

\usepackage{pageslts}
\NewEnviron{tee}{\BODY\typeout{Marker Tee [start] ^^J \BODY ^^JMaker Tee [end]}}

\input{preamble/math}
\newcommand{\real}[1]{\mathbb{R}^{#1}{}}

\newcommand{\bmat}[1]{\begin{bmatrix}#1\end{bmatrix}}

\newcommand{\transpose}{^\mathrm{T}}

\newcommand{\inverse}{^{-1}}

\newcommand{\defeq}{\doteq}

\DeclarePairedDelimiter{\norm}{\lVert}{\rVert}

\newcommand{\vct}[1]{\mathbf{#1}}

\DeclareMathOperator{\blkdiag}{blkdiag}

\DeclareMathOperator{\trace}{tr}

\DeclareMathOperator{\stack}{stack}

\DeclareMathOperator{\expect}{\mathbb{E}}

\newcommand{\subjectto}{\textrm{subject to}\;}

\providecommand{\vy}{\vct{y}}

\providecommand{\vA}{\vct{A}}

\providecommand{\vB}{\vct{B}}

\providecommand{\vK}{\vct{K}}

\providecommand{\vY}{\vct{Y}}

\providecommand{\vtheta}{\bm{\theta}}

\providecommand{\vSigma}{\bm{\Sigma}}

\providecommand{\cC}{\mathcal{C}}

\providecommand{\cK}{\mathcal{K}}

\providecommand{\cN}{\mathcal{N}}

\providecommand{\cX}{\mathcal{X}}

\usepackage{units}

  \newcommand{\newcolorlabel}[2]{%
  \expandafter\newcommand\csname #1\endcsname[1]{%
    \tikz[baseline]{\node[text=white,fill=#2,anchor=base,text height=1.3ex,text depth=0.1ex,font=\sffamily\bfseries]{##1}}}%
}

\newcommand{\newcommenter}[2]{%
  \expandafter\newcommand\csname #1\endcsname[1]{%
    \fcolorbox{#2}{#2}{\color{white}\textsf{\textbf{#1}}}
    {\color{#2}##1}}%
  \expandafter\newcommand\csname at#1\endcsname{%
    \fcolorbox{#2}{#2}{\color{white}\textsf{\textbf{@#1}}}
    {\color{#2}}}%
  \expandafter\newcommand\csname #1cite\endcsname[1]{%
    \csname #1\endcsname {[##1]}
  }%
  \expandafter\newcommand\csname #1ref\endcsname[1]{%
    \csname #1\endcsname {$\blacktriangleright$##1}
  }%
  \expandafter\newcommand\csname #1hl\endcsname[2]{%
    \colorbox{#2}{\color{white}\textsf{\textbf{#1}}}\sethlcolor{Azure2}\hl{##2}~%
    \expandafter\ifx\csname commentarrow\endcsname\relax$\leftarrow$\else \commentarrow[#2]\fi~%
    {\color{#2}##1}}%
  \expandafter\newcommand\csname #1st\endcsname[2]{%
    \colorbox{#2}{\color{white}\textsf{\textbf{#1}}}\sout{##2}~%
    \expandafter\ifx\csname commentarrow\endcsname\relax$\leftarrow$\else \commentarrow[#2]\fi~%
    {\color{#2}##1}}%
}
\newcommenter{TODO}{DodgerBlue1}
\newcommenter{rtron}{Green3}

\usepackage{comment}

\usepackage{pdfcomment}

\usepackage{soul}

\usepackage[normalem]{ulem}

\usepackage{csquotes}

\usepackage{tikz}
\usetikzlibrary{calc}
\usetikzlibrary{matrix}
\usetikzlibrary{chains,scopes}
\usetikzlibrary{shapes.geometric}
\usetikzlibrary{arrows.meta}
\usetikzlibrary{decorations.markings}
\usetikzlibrary{decorations.pathreplacing}
\usetikzlibrary{backgrounds}

\tikzset{
  dim above/.style={to path={\pgfextra{
        \pgfinterruptpath
        \draw[>=latex,|->|] let
        \p1=($(\tikztostart)!1.5em!90:(\tikztotarget)$),
        \p2=($(\tikztotarget)!1.5em!-90:(\tikztostart)$)
        in(\p1) -- (\p2) node[pos=.5,sloped,above]{#1};
        \endpgfinterruptpath
      }
    }
  },
  dim double above/.style={to path={\pgfextra{
        \pgfinterruptpath
        \draw[>=latex,|->|] let
        \p1=($(\tikztostart)!3em!90:(\tikztotarget)$),
        \p2=($(\tikztotarget)!3em!-90:(\tikztostart)$)
        in(\p1) -- (\p2) node[pos=.5,sloped,above]{#1};
        \endpgfinterruptpath
      }
    }
  },
  dim below/.style={to path={\pgfextra{
        \pgfinterruptpath
        \draw[>=latex,|->|] let 
        \p1=($(\tikztostart)!-1em!-90:(\tikztotarget)$),
        \p2=($(\tikztotarget)!-1em!90:(\tikztostart)$)
        in (\p1) -- (\p2) node[pos=.5,sloped,below]{#1};
        \endpgfinterruptpath
      }
    }
  },
}

\tikzset{
    right angle quadrant/.code={
        \pgfmathsetmacro\quadranta{{1,1,-1,-1}[#1-1]}     
        \pgfmathsetmacro\quadrantb{{1,-1,-1,1}[#1-1]}},
    right angle quadrant=1, 
    right angle length/.code={\def\rightanglelength{#1}},   
    right angle length=2ex, 
    right angle symbol/.style n args={3}{
        insert path={
            let \p0 = ($(#1)!(#3)!(#2)$) in     
                let \p1 = ($(\p0)!\quadranta*\rightanglelength!(#3)$), 
                \p2 = ($(\p0)!\quadrantb*\rightanglelength!(#2)$) in 
                let \p3 = ($(\p1)+(\p2)-(\p0)$) in  
            (\p1) -- (\p3) -- (\p2)
        }
    }
}

\newcommand{\pgfextractangle}[3]{%
    \pgfmathanglebetweenpoints{\pgfpointanchor{#2}{center}}
                              {\pgfpointanchor{#3}{center}}
    \global\let#1\pgfmathresult  
}

\usetikzlibrary{shapes.arrows}
\newcommand{\commentarrow}[1][Azure4]{\tikz[baseline=-3pt]{\node[shape border uses incircle, fill=#1,rotate=180,single arrow, inner sep=1pt, minimum size=6pt, single arrow head extend=2pt]{};}}

\tikzset{ax/.style={-latex,line width=2pt}}

\tikzset{camera/.style={fill=Sienna1,fill opacity=0.5},%
image plane/.style={draw=RoyalBlue3,line width=2pt}}

\newcommand{\astar}{$\mathtt{A^*}$}

\usepackage{float}
\usepackage{balance}

\title{\LARGE \bf
  \ifarxiv
    Technical Report on Optimal Linear Multiple Estimation for Landmark-Based Planning via Control Synthesis
  \else
    Optimal Linear Multiple Estimation\\for Landmark-Based Planning via Control Synthesis
  \fi
}

\author{Chenfei Wang and Roberto Tron
  \thanks{The authors are with the Department of Mechanical Engineering,
    Boston, MA, 02215 USA. Email:
    {\tt\small \{wang1029,tron\}@bu.edu}.}%
  \thanks{This work was partially supported by ONR MURI award N00014-19-1-2571.}%
}

\begin{document}

\maketitle
\thispagestyle{empty}
\pagestyle{empty}

\begin{abstract}
  A common way to implement navigation in mobile robots is through the use of landmarks. In this case, the main goal of the controller is to make progress toward a goal location (stability), while avoiding the boundary of the environment (safety). In our previous work, we proposed a method to synthesize global controllers for environments with a polyhedral decomposition; our solution uses a Quadratically Constrained Quadratic Program with Chance Constraints to take into account the uncertainty in landmark measurements. 
  Building upon this work, we introduce the concept of virtual landmarks, which are defined as linear combinations of actual landmark measurements that minimize the uncertainty in the resulting control actions. Interestingly, our results show that the first minimum-variance landmark is independent of the feedback control matrix, thus decoupling the design of the landmark from the one of the controller; attempting to derive additional, statistically independent landmarks, however, requires solving non-convex problems that involve also the controller.
  Numerical experiments demonstrate that, by minimizing the variance of the inputs, the resulting controller will be less conservative, and the robot is able to complete navigation tasks more effectively (faster and with less jitter in the trajectory).
\end{abstract}

\section{INTRODUCTION}
Motion planning is a fundamental problem for mobile robotics. A basic approach \cite{lavalle2006planning, schurmann2009computational,choset2005principles} consists of abstracting the environment into a graph (e.g., via a cell decomposition or random sampling), and then dividing the planning problem into a high-level planning on the vertices of the graph (i.e., sequences of regions) followed by a low-level feedback controller for transitions (between regions). 
In this setting, the high-level planning component is relatively simple, and can be completed, for instance, with the celebrated \astar{} algorithm. The low-level controller, however, needs to balance multiple requirements and challenges: making progress toward the next step of the high-level plan toward a goal (stability), avoiding the environment boundary (safety), reduce the effect of noise in the measurements (robustness to noise). 
One way to address this set of requirements is by focusing on linear feedback laws, and perform a worst-case analysis to synthesize controllers that are stable and safe for given probabilistic bounds on the noise in the system, as done in our previous work \cite{ccbf}.
This approach results in a minimal amount of computations at runtime (which simply involves applying the pre-computed linear controller), but, by its nature, will result in conservative control actions.

In this paper, however, we make the following observations:
\begin{lenumerate}{O}
\item\label{it:obserbation_first} We can shape the uncertainty of feedback control actions (to a certain degree) through different linear combinations of the inputs.
\item The probability of violating a safety constraint, i.e., moving toward a wall, depends on the uncertainty of the control along the normal to the wall (moving parallel to the wall does no increase the risk).
\item\label{it:obserbation_last} In practice, statistically consistent estimators (e.g., Kalman filters \cite{kalman}) include a covariance matrix for the uncertainty in the estimate at each time step; this information could be used in real time to reduce the conservatism of a worst-case analysis.
\end{lenumerate}

\subsection{Background and Previous work}
Stability and safety are crucial for control system. Control Lyapunov Functions (CLFs,~\cite{ames2014rapidly,sontag1989universal}) and Control Barrier Functions (CBFs,~\cite{c1,c1b}) have become a popular methods to create controllers that ensure, respectively, user-defined convergence guarantees and forward invariance of safety sets; this popularity is due to the flexibility they provide (they only require mild conditions on the differentiability of the barrier functions and can be applied to systems with high relative degree~\cite{galloway2015torque,xiao2021high,nguyen2016exponential}), and the fact that can be implemented relatively efficiently, as they generally result in quadratic programs (QPs).
While typically employed for real-time control, CLFs and CBFs have also been applied to offline path planning. In particular, the present paper follows the paradigm introduced by the work of \cite{Bahreinian:ACC21,Bahreinian:IROS21}, which uses linear CLF and CBF functions to design linear output feedback controllers that can guide the robot in a convex region, using high-level discrete path planning to determine a sequence of regions toward the goal. In our previous work \cite{WangC:ACC22}, we have extended the same paradigm to the case where the measurements are affected by Gaussian noise with known bounds on the covariance. In this case, the CLF and CBF constraints become stochastic quantities that can be transformed into \emph{chance CLF and CBF constraints} (ensuring that the original CLF and CBF constraints hold with a user-defined probability) using Chebyshev's bound; this formulation leads to solving Quadratic Constraint Quadratic Programings (QCQPs) instead of QPs.
Although the resulting controllers are a significant improvement with respect to other work that does not explicitly consider noise (e.g., \cite{conner2003composition,habets2006reachability,belta2005discrete}), they can still be overly conservative: for instance, \cite{WangC:ACC22} does not provide a way to incorporate better bounds if these become available at runtime.

The effect of noise (and other disturbances) can be also considered on the controls (i.e., the controller outputs a nominal $u$, but the system receives $u$ plus a disturbance $d$); this can be taken into account again by using a worst-case analysis leading to robust QP formulations \cite{dawson2022safe,choi2021robust,garg2021robust,jankovic2018robust,seiler2021control,nguyen2021robust} or by considering steady-state-errors~\cite{kolathaya2018input}. Such line of work, however, considers disturbances affecting the \emph{output} of the controller (Fig.~\ref{fig:deviation-control}), while we consider noise affecting the \emph{input} of the controller (Fig.~\ref{fig:deviation-measurements}); the two settings are somewhat complementary, and we will focus on the latter.

\begin{figure}
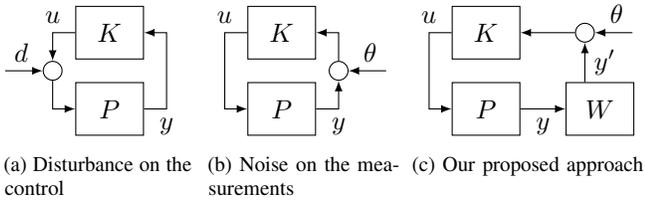

  \centering
  \subfloat[Disturbance on the control]{\begin{tikzpicture}
  \input{figure/tikz/feedback-common}
  \node[junction] (e) at ($(C out)!0.5!(P in)$) {};
  \begin{scope}[connection]
  \draw (P.east) -- (P out) -- (C in) -- (C.east);
  \draw (C.west) -- (C out) -- (e);
  \draw (e) -- (P in) -- (P.west);
  \draw ($(e.west)+(-5mm,0)$) node[anchor=south west] {$d$} -- (e);
\end{scope}
\end{tikzpicture}

\label{fig:deviation-control}}
  \hspace{1mm}
  \subfloat[Noise on the measurements]{\begin{tikzpicture}
  \input{figure/tikz/feedback-common}
  \node[junction] (e) at ($(P out)!0.5!(C in)$) {};
  \begin{scope}[connection]
    \draw (P.east) -- (P out) -- (e);
    \draw (e) -- (C in) -- (C.east);
    \draw (C.west) -- (C out) -- (P in) -- (P.west);
    \draw ($(e.east)+(5mm,0)$) node[anchor=south east] {$\theta$} -- (e);
  \end{scope}
\end{tikzpicture}

\label{fig:deviation-measurements}}
  \hspace{1mm}
  \subfloat[Our proposed approach]{\begin{tikzpicture}
  \input{figure/tikz/feedback-common}
  \node[feedback box,right=6mm of P] (W) {$W$};
  \node[junction] (e) at ($(C in-|W)+(-2mm,0)$) {};
  \begin{scope}[connection]
    \draw (P.east) -- (W.west);
    \draw (e) -- (C.east);
    \draw (C.west) -- (C out) -- (P in) -- (P.west);
    \draw ($(e.east)+(5mm,0)$) node[anchor=south east] {$\theta$} -- (e);
    \draw (e|-W.north) -- (e) node[pos=0.5, anchor=west]{$y'$};
  \end{scope}
\end{tikzpicture}

\label{fig:deviation-ours}}
  \caption{Block diagram of different types of disturbances in feedback loops between a plant $P$ and a controller $K$, and our proposed approach that uses a linear combination $W$ of the measurements $y$ to produce virtual landmarks $y'$.}
  \label{fig:blocks}
\end{figure}

\subsection{Paper Contributions}

In this work, we build upon observations \ref{it:obserbation_first}--\ref{it:obserbation_last} to enhance our prior work \cite{ccbf}. We first discuss a novel way to use the linearity of the controllers to mix the different input landmark measurements (Section \ref{section2}).
Then, we show that we can reduce the conservatism of the chance CBF constraints by creating virtual landmarks that minimize the variance in the direction of the constraints (Sec.~\ref{section4}). We show that it is possible to generate at most $N_y$ uncorrelated virtual marks from $N_y$ original landmarks (Sec.~\ref{sec:LME}). Interestingly, our analysis shows that the first virtual landmark can be selected independently of the linear feedback controller and the safety constraints, while the subsequent ones require solving a joint non-convex problem over both the combination weights and the controller. The first virtual landmark can be computed by from the covariance matrices using simple matrix operations. Hence, we propose therefore the following strategy: offline, we use the known covariance bounds to merge all landmarks into a single virtual landmark (Fig.~\ref{fig:deviation-ours}), and synthesize controllers based on the tightened chance constraints; online, if updated covariance matrices are available, we can recompute the virtual landmark and update the controller with simple matrix operations. Both steps reduce the conservatism of the control actions.

We show via numerical simulations that the reduced conservatism translates in a final trajectory that converges faster and with less jitter.


\section{Problem Formulation and Preliminaries}
\label{section2}

\subsection{Dynamical System}
We model the robot as a linear time-invariant system
\begin{equation}\label{dyn}
  \dot{x} =A x+B u,
\end{equation}
where $A \in \mathbb{R}^{n \times n}$, and $B \in \mathbb{R}^{n \times m}$ are assumed to form a stabilizable pair, $x\in \mathbb{R}^{n}$ is the state, and $u\in \mathbb{R}^m$ is the control input (the output measurement model is also linear and is discussed in \S\ref{sec:measurement model}).
Our formulation heavily relies on the linearity of model \eqref{dyn}. However, it can be applied to nonlinear robots by either linearizing the model (note that each environment cell can use a different linearization point) or by using a lower-level nonlinear controller (e.g., see \cite{Bahreinian:IROS21} for a practical demonstration of this approach).

\begin{figure}[b]
  \centering
  \includegraphics[scale=0.2]{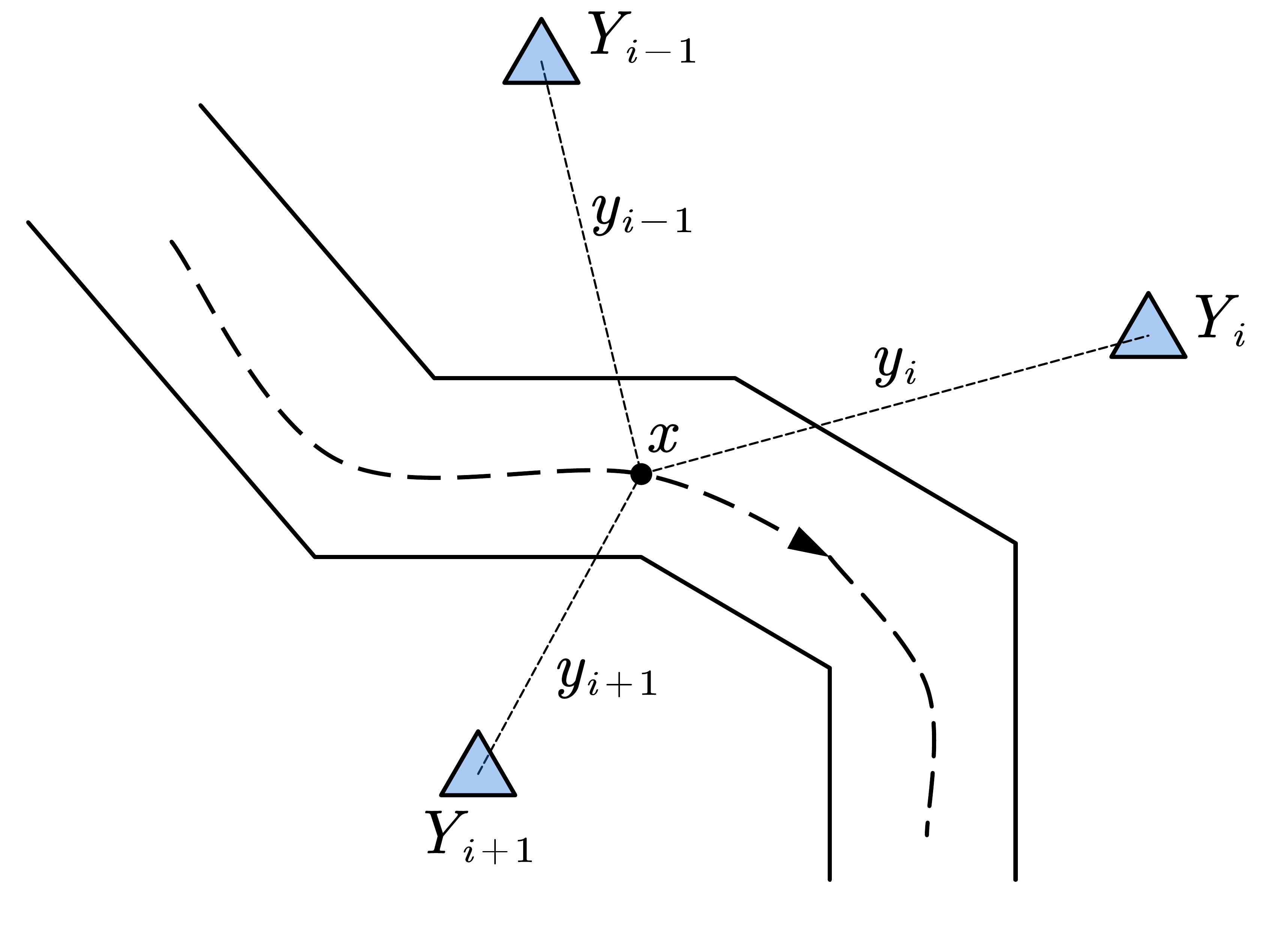}
  \caption{Landmark based robot navigation: the robot can measure the relative location of landmarks. And it utilizes this information to complete the control and navigation task.}\label{fig:1}
\end{figure}

\subsection{The Environment and Decomposition-Based Navigation}
It is assumed that the free space in the environment can be decomposed into convex cells (polytopes). As in~\cite{Bahreinian:ACC21}, we tackle the path planning problem by first executing a high level path planning strategy, which identifies a sequence of cells from the one that contains the starting location to the one that contains the goal location. This is typically implemented by first representing the cells and the possible motions between them as the vertices and edges of a graph, respectively, and then using a discrete path planning algorithm such as \astar. Then, we design controllers for each cell $\cX$ in the decomposition that use measurements with respect to fixed landmarks, as explained in the two sections below. Each cell $\cX$ is defined by a set of linear inequalities
\begin{equation}\label{eq:polytope}
  \cX=\{x\in\real{n}: A_xx+b_x\leq 0\},
\end{equation}
where each row of the system $A_xx=b_x$ defines the hyperplane of one of the faces of the polytope.

\subsection{Landmark Measurement Model}
\label{sec:measurement model}
We assume that, for each cell $\cX$, the locations of $N_y\geq 2$ landmarks $\{Y_i\}_{i=1}^L$ are known at planning time (if $N_y=1$ it is still possible to solve the path planning problem, but there is no opportunity to apply the main results of this paper discussed in Sec.~\ref{sec:single landmark}).
We assume that the robot is able to measure the relative displacements (see Fig.~\ref{fig:1})
\begin{equation}\label{eq:measurements}
  y_i(x)=Y_i-x+\theta_i,
\end{equation}
with noise $\theta_i\sim \cN(0,\Sigma_i)$. We assume that each measurement $y_i$ is statistically independent from other measurements $y_j$, $j\neq i$. In practice, the measurements can be acquired using typical hardware installed on mobile robots, such as LIDARs or RGBD cameras, possibly integrating multiple measurements over time using SLAM \cite{Eade:CVPR06} or filtering techniques \cite{Prazenica:GNCC05}.
To simplify the notation for the derivations in the next section, we write~\eqref{eq:measurements} in vector form as $\vy=\vY+\vtheta$, where $\vy\defeq \stack(\{y_i\})$, $\vY\defeq \stack(\{Y_i\})$, and $\vtheta\defeq \stack(\{\theta_i\})$.

Since the location of the landmarks is known and the measurement model \eqref{eq:measurements} is known, it is possible to simulate the measurement of landmarks with arbitrary locations, as detailed in the following.
\begin{definition}\label{def:virtual landmark}
  We define a \emph{virtual landmark} as an arbitrary location $Y'$ obtained as a linear combination of the existing landmarks, i.e., $Y'=\sum_i^{N_y} W_i Y_i$ where $\sum_i^{N_y} W_i=I$.
\end{definition}
\begin{proposition}\label{prop:virtual landmark}
  With the notation above, the \emph{virtual measurement} $y_W=Y_W-x$ corresponding to the virtual landmark $Y_W$ is computed as a linear combination of the measurements, $y_W=\sum_i^{N_y} W_i y_i+\theta_W$, where $\theta_W=\sum_i^{N_y} W_i\theta_i$ is a noise vector with distribution $\cN(0,\Sigma_W)$, where $\Sigma _W=\sum_i^{N_y}{W_i\Sigma _iW_{i}\transpose }$.
\end{proposition}
\begin{proof}
  The first part of the claim follows from the definitions of $y_W$ and $Y_W$ and the distributivity property: $y_W=\sum_i^{N_y} W_i y_i=\sum_i^{N_y} W_i Y_i +\sum_i^{N_y} W_i x+\sum_i^{N_y} W_i\theta_i=Y_W-x+\theta_W$. The second part of the claim derives from the fact that the mean of $\theta_W$ is $\expect[\theta_W]=\sum_{i=1}^{N_y}W_i\expect[\theta_i]=0$, and its covariance is $\expect[\theta_W\theta_W\transpose]=\sum_{i=1}^{N_y}\sum_{j=1}^{N_y}W_i\expect[\theta_i\theta_j\transpose]W_j\transpose=\sum_{i=1}^{N_y}W_i\expect[\theta_i\theta_i\transpose]W_i\transpose=\sum_i^{N_y}{W_i\Sigma _iW_{i}\transpose }$, since $\theta_i,\theta_j$ are statistically independent for $i\neq j$.
\end{proof}

\subsection{Landmark-Based Linear Controllers}
Once the sequence of cells to traverse is established, we need a local controller that can guide the robot during the transition from a cell to the following one. In our approach we define a linear feedback controller that combines the multiple measurements $\{y_i\}$ according to
\begin{align}\label{eq:controller}
  u=\sum_i^{N_y}K_{i}y_i+k=\vK\vY+k,
\end{align}
where $K_i$ is the feedback matrix for measurement $y_i$, $k$ is a common control bias term, and $N_y$ is the number of measurements; additionally, the controller is written in vectorized form using $\vK\defeq\bmat{K_1 &\ldots & K_{N_y}}$. Note that~\eqref{eq:controller} represents an \emph{output feedback controller}, because we do not measure the state $x$ directly, but $N_y$ affine versions of it, each one with, in general, different noise characteristics.

We now introduce a preliminary result which allow to compute~\eqref{eq:controller} using different expressions that, in the absence of noise, are all equivalent.

\begin{proposition}[Modified from~\cite{Bahreinian:IROS21}]\label{prop:K remixing}
  Let $Y_W\in\real{d}$ be a virtual landmark. Then, the control $u$ in~\eqref{eq:controller} can be computed from the virtual measurement $y_W=Y_W-x$ as
  \begin{equation}\label{eq:controller virtual}
    u=K_Wy_W+k_W,
  \end{equation}
  where $K_W\defeq\sum_i^{N_y}K_i$ and $k_W\defeq\sum_i^{N_y}K_i(Y_i-Y_W)$.
\end{proposition}
See \cite{Bahreinian:IROS21} for a proof. The proposition, in principle, allows us to compute the control $u$ based on any point in the environment (and in particular with respect to arbitrary virtual landmarks).

\subsection{Control Barrier Function}
Control Barrier Functions (CBFs)  are a widely used tool in safety critical system.

\begin{definition}[ECBF, \cite{c3}] For an affine nonlinear system $\dot{x}=f\left( x \right) +g\left( x \right) u$ and a set $\cC=\{x\in \mathbb{R}^n : h(x)\ge0\}$, where the system has relative degree $r$ with respect to the function $h(x)$, $h(x)$ is an \emph{Exponential Control Barrier Function } (ECBF \cite{c3}), if there exists $\alpha_h\in \mathbb{R}^{1\times r}$ such that
  \begin{equation} \label{ecbf}
    L^r_fh\left( x \right) +L_gL^{r-1}_fh\left( x \right) u+ \alpha_h \xi(x) \ge 0,
  \end{equation}
  where $\xi \left( x \right) =\left[ h\left( x \right) ,L_fh\left( x \right) ,...,L_{f}^{r-1}h\left( x \right) \right]\transpose$ and $\alpha_h=[\alpha_{h0}, \alpha_{h1}, ..., \alpha_{hr-1}]$. The row vector $\alpha_h$ is selected satisfying the requirements given by \cite{c3}.
\end{definition}

In the context of a robot navigation task, a wall or boundary of a convex region is represented with the barrier function $h(x)=A_hx+b_h$, where $A_h \in \mathbb{R}^{1\times n}$ represents the normal to the wall, $b_h \in \mathbb{R}$, and the set $C=\{x\in\real{n}:h(x)>0\}$ represents the free space outside the wall. Considering the LTI system \eqref{dyn}, inequality \eqref{ecbf} can be written in matrix form as
\begin{equation} \label{lecbf}
  \vA  x+\vB  u+d\ge 0,
\end{equation}
where
\begin{equation} \notag
  \begin{aligned}
    \vA  &\defeq \alpha_h \bmat{
      A_h\\
      A_hA\\
      \vdots\\
      A_hA^{r-1}\\
      A_hA^r\\
    },    &
    \vB  &\defeq A_hA^{r-1}B, & d\defeq \alpha _{h0}b_h.
  \end{aligned}
\end{equation}

\begin{remark}
  For simplicity, we present our theory using a single CBF constraint. However, it is straightforward to extend the results to multiple the general case with multiple walls by using additional constraints of the form \eqref{lecbf}.
\end{remark}

\subsection{Control Barrier Function Chance Constraints}
In practice, the control output $u$ for constraint~\eqref{lecbf} is calculated from the measurements as in~\eqref{eq:controller} (or~\eqref{eq:controller virtual}, when virtual landmarks are used), which contains noise and is hence a stochastic quantity. As a consequence, the inequality constraint~\eqref{lecbf} becomes a stochastic quantity as well: we therefore rewrite it as a chance constraint

\begin{equation}\label{eq:chance constraint}
  P(\vA  x+\vB  u+d\ge 0) \ge 1 - \eta(x),
\end{equation}
where $P$ represents a Gaussian distribution, and $\eta(x)$ is a function of the state $x$ satisfying the following conditions:
\begin{itemize}
\item $0<\eta(x)<1$ for $x \in \cX$,
\item $\eta(x)$ is continuous and concave.
\end{itemize}

At this point, we could use the results of~\cite{Wang:ACC21} to transform~\eqref{eq:chance constraint} into constraints that are quadratic in $\vK$ and linear in $x$ (so that they can be used to find the controllers); however, we propose here a derivation based on the Gaussian cumulative distribution function that is both simpler and leads to tighter constraints.

\begin{proposition}\label{prop:chance constraint}
  A sufficient condition for~\eqref{eq:chance constraint} to hold is
  \begin{equation} \label{sufficient_condition}
    \bar{K}_2\Gamma \bar{K}_2\transpose -\bar{K}_1x-\bar{k}_3\le 0
  \end{equation}
  where $\eta_0$ is a lower bound on $\eta$ in the current region~$\cX$, $\Gamma =\frac{1}{\eta_0}\vSigma$, and
  \begin{align}
    \bar{K}_1&\defeq \vA  -\vB  \sum_i K_i,\\
    \bar{K}_2&\defeq \vB \vK,\\
    \bar{k}_3&\defeq \vB (\vK \vY+k)+d.
  \end{align}
\end{proposition}

Before proving the proposition, we need the following:
\begin{lemma}\label{lemma:bound inverse cumulative}
  Let $\Phi(\cdot)$ be the cumulative distribution of a scalar Gaussian random variable with zero mean and unit variance. The following inequality holds for $\eta>0$:
  \begin{equation}\label{eq:cumulative inequality}
    \frac{1}{\eta} > \Phi^{-1}(1-\eta)
  \end{equation}
\end{lemma}
\begin{proof}
  Consider the function $f(\eta)=\Phi(\frac{1}{\eta}) + \eta - 1$, $\eta>0$. We have $\lim_{\eta \to 0^+} f(\eta)=0$ and $\frac{\mathrm{d}}{\mathrm{d}\eta}f(\eta)=1-\frac{2e^{-\frac{1}{\eta^2}}}{\sqrt{\pi}\eta^2} \geq 1-\frac{2}{e\sqrt{\pi}} > 0$. Thus, we get $f(\eta)=\Phi(\frac{1}{\eta}) + \eta - 1 > 0$, from which the claim follows.
\end{proof}

\begin{proof}[of Prop.~\ref{prop:chance constraint}]
  Substituting~\eqref{eq:controller} into~\eqref{eq:chance constraint} we obtain
  \begin{equation}\label{eq:chance constraint K}
    P(\bar{K}_1x+\bar{K}_2\theta +\bar{k}_3\ge 0)\ge 1-\eta (x).
  \end{equation}
  Letting $\vSigma=\blkdiag(\{\Sigma_i\})$, we have that the random variable $\frac{1}{\sqrt{\bar{K}_2\vSigma\bar{K}_2\transpose}}\bar{K}_2\vtheta$ has zero mean and unit variance. From the definition of cumulative distribution, we then have that \eqref{eq:chance constraint K} is equivalent to
  \begin{equation} \label{eq:original condition}
    \frac{\bar{K}_1x+\bar{k}_3}{\bar{K}_2\vSigma \bar{K}_2\transpose }\ge \Phi ^{-1}\left( 1-\eta(x) \right).
  \end{equation}

  Furthermore, from Lemma~\ref{lemma:bound inverse cumulative} and the definition of $\eta_0$, we have
  \begin{equation}\label{eq:bound eta}
    \frac{1}{\eta_0}\geq \frac{1}{\eta(x)}>\Phi ^{-1}\left( 1-\eta(x) \right).
  \end{equation}

  Combining~\eqref{eq:original condition} with~\eqref{eq:bound eta} and rearranging terms, we obtain the result of the proposition
\end{proof}


\begin{remark}\label{remark:K opt covariance}
  The term $\sigma_K^2=\bar{K}_2\vSigma \bar{K}_2\transpose $ indicates the variance of the controller towards the direction of the barrier boundary. From~\eqref{eq:original condition}, we see that reducing $\sigma_K^2$ makes the chance constraint looser, hence giving the opportunity to improve $\vK$. In Sec.~\ref{section3}, we show how to minimize this variance $\sigma_K^2$ (and hence improve $\vK$) by weighting the measurements.
\end{remark}

We can apply \eqref{sufficient_condition} as a constraint while optimizing, e.g., the Frobenious norm of the feedback matrix, leading to the optimization problem
\begin{equation}\label{eq:K opt}
  \begin{aligned}
    \min & \norm{\vK}^2_2 \\
    \subjectto& \bar{K}_2\Gamma \bar{K}_2\transpose -\bar{K}_1x-\bar{k}_3\le 0\\
    & \forall x \in \cX
  \end{aligned}
\end{equation}

The optimization can be rewritten as a QCQP (see Appendix~\ref{sec:appendix K opt} for a detailed derivation).
\begin{equation}\label{eq:K opt dual}
  \begin{aligned}
    \min  &\norm{\vK}_2\\
    \subjectto &  \bar{K}_2\Gamma \bar{K}_2\transpose -\bar{k}_3+b_x\transpose\lambda \le 0\\
    &  A_x\transpose\lambda+\bar{K}_1\transpose=0\\
    & \lambda \ge 0\\
  \end{aligned}
\end{equation}

An analogous process (omitted here for the sake of brevity), can be applied to obtain Control Lyapunov Function chance constraints that can be incorporated in the optimization problem \eqref{eq:K opt dual}: it is sufficient to replace $A_h$, $b_h$ with the definition of the hyperplane identifying the exit face of the cell, inverting the sense of the inequality in \eqref{lecbf}, and then follow the same steps.


\section{Virtual Landmarks for Minimum Chance Control Synthesis} \label{section3}

In this section we propose to improve the feasibility of the QCQP~\eqref{eq:K opt dual} by replacing the measurements and landmarks $\{y_i,Y_i\}$ with virtual landmarks. We first consider the case of a single virtual landmark (for which the optimal weights $\{W_i\}$ can be obtained independently of the controller $\cK$ and the wall parameters $A_h,b_h$), and then the case of multiple, statistically independent landmarks (which require solving for the weights $\{W_i\}$ and the controller $\cK$ simultaneously).

\subsection{Single Virtual Landmark}\label{sec:single landmark}
If we use a single measurement $y_W$ with respect to a single virtual landmark $Y_W$ (see Def.~\ref{def:virtual landmark}), the optimization problem~\eqref{eq:K opt dual} is still valid, except that $\bar{K}_1$, $\bar{K}_2$ and $\bar{k}_3$ need to be changed; as a consequence, from the characterization of the covariance of $y'$ in Prop.~\ref{prop:virtual landmark}, we have that $\sigma_K^2=\bar{K}_2\Sigma_W\bar{K}_2\transpose$. Following Remark~\ref{remark:K opt covariance} we the propose to find the set of weights $\{W_i\}$ that maximize the feasibility of the QCQP \eqref{eq:K opt dual}; specifically, we aim to solve the problem
\begin{equation} \label{op1}
  \begin{array}{c}
    \min_i \quad \bar{K}_2\Sigma _W\bar{K}_2\transpose\\
    s.t.\quad \sum_i{W_i=I},\\
  \end{array}
\end{equation}

\begin{proposition}\label{prop:weights first}
  The solution of the optimization problem \ref{op1} can be obtained with simple matrix operations:
  \begin{equation} \label{solu1}
    W_i=\Bigl( \sum_i{\Sigma _{i}^{-1}} \Bigr)\inverse\Sigma _{i}^{-1}.
  \end{equation}
\end{proposition}
The proof of this proposition is given in Appendix~\ref{sec:appendix weights}, and is related to standard proofs for Weighted Least Squares estimation.
\begin{remark}
  The matrix $\bar{K}_2$ does not appear in \eqref{solu1}. Thus, the weight $\hat{W}_i$ is independent with the controller $K$.
  This shows that the control design can be factorized into optimal estimation and control, which is reminiscent of LQG design. 
\end{remark}
\begin{remark} It is possible to prove that $\sigma_K^2$ cannot become worse after fusing the measurements into a single landmark, although such proof has been omitted due to space constraints.
\end{remark}

\subsection{Linear Multiple Estimation}
\label{sec:LME}
In the previous section we have shown how to create a single virtual landmark from the initial $N_y$ landmarks to improve the feasibility of the design problem for the controller $\cK$; the solution, as it turns out, does not depend on the actual controller chosen. Intuitively, however, the use of a single landmark may loose some information due to the reduction in dimensionality. It is then natural to ask whether it is possible to find additional virtual landmarks that may provide complementary information (more precisely, that are statistically uncorrelated). In this section we show that this is indeed possible, although the result depends on the controller (which needs to be synthesized jointly).

Inspired by the Principal Component Analysis (PCA) which decomposes the original data to different directions, we propose a similar decomposition-based weighting method.

Suppose we want to generate a set of a set of additional virtual measurements $l>1$.
\begin{equation}
  y_{W,l}=\sum_iW_i^ly_i,
\end{equation}
where $y_{Wk}$ indicates the $k$-th weighted measurement. To ensure that each weighted measurement carries novel information, we require that it is uncorrelated with previous virtual landmarks by enforcing the constraint
\begin{equation} \label{uncor_constraint}
  \expect [y_l'y_j']=0.
\end{equation}

For $l=1$, $\{W_i^{1}\}$ is given by \eqref{solu1}; $y'_1=\sum_iW_i^1y_i$, our first virtual landmark, can be interpreted as the first principal component. For $l\ge 2$, $\{W_i^{1}\}$ needs to follow the uncorrelated constraint \eqref{uncor_constraint}, which can be rewritten as the following form,
\begin{equation} \label{uncor_constraint2}
  \sum_i{\left( W_{i}^{j} \right)\transpose\Sigma _iW_{i}^{l}=0}
\end{equation}

To show this, suppose that the constraint \eqref{uncor_constraint2} hold. For two virtual landmarks $j$ and $l$, the noise terms are $\delta_j = \sum_i{W^{j}_i\theta_i}$ and $ \delta_l = \sum_i{W^{l}_i\theta_i}$. Requiring that $y_j$ and $y_l$ are uncorrelated can be formally stated as:
\begin{equation}
  \begin{aligned}
    \textrm{Cov}(\delta _j,\delta _l)&=\mathbb{E}\left[ \delta _{j}\delta_l\transpose  \right]\\
                                     &=\mathbb{E} \left[ \left( \sum_i{W_{i}^{j}\theta _i} \right) \left( \sum_i{W_{i}^{l}\theta _i} \right)\transpose\right]\\
    &=\sum_i{\sum_{i^\prime}{W_{i}^{j}\mathbb{E} \left( \theta _i\theta _{i^\prime}\transpose  \right) W_{i^\prime}^{lT}}}\\
    &=\sum_i{W_{i}^{j}\Sigma _iW_{i}^{lT}}=0,
  \end{aligned}
\end{equation}

With the constraints above, we can find each new landmark $y_{l}'$ in sequence after solving for $y'_1,\ldots,y'_{l-1}$ by solving the following optimization problem (which is similar to \eqref{op1} with the additions of the uncorrelation constraints \eqref{uncor_constraint2}).
\begin{equation} \label{op2}
  \begin{array}{c}
    \min_{\{W_{i}^{l}\}} \bar{K}_2\sum_i{(}W_{i}^{l}\Sigma_iW_{i}^{lT})\bar{K}_2\transpose\\
    s.t.\begin{cases}
      \sum_i{W_i^l=I}\\
      \sum_i{\left( W_{i}^{1} \right)\transpose\Sigma _iW_{i}^{l}=0},\\
      \qquad \vdots \\
      \sum_i{\left( W_{i}^{l-1} \right)\transpose\Sigma _iW_{i}^{l}}=0,\\
    \end{cases}\\
  \end{array}
\end{equation}

The solution is given in Appendix~\ref{sec:appendix LME}.

The weights for virtual landmarks other than the first, i.e., $l>1$, depend on $\bar{K}_2$, which contains the controller $K$ (as shown in \eqref{op2} and already mentioned above). Hence, to rigorously make use of additional virtual landmarks, we would need to solve \eqref{eq:K opt dual} jointly with the weights $W_i^l$, under the constraints from \eqref{op2}. In general, this results in a non-convex formulation, and whose solution is beyond the scope of this paper. As a result, we will limit the simulations to the case of a single landmark, $l=1$, showing that the this already significantly improves the performance of the controller.



\subsection{Control Synthesis} \label{section4}

As discussed above, for each chance constrain (i.e., for each wall), we can pick virtual landmarks that minimize the variance in each chance CBF constraint. To actually synthesize the controller $K$, we therefore propose the following steps:

\begin{enumerate}
\item Generate virtual landmarks $y_{W,l}$.
\item Solve the controller design problem~\eqref{eq:K opt dual} based on the virtual landmarks.
\item If the noise covariance information is updated, use simple matrix operations to recompute the weights (Prop.~\ref{prop:weights first} or eq.~\ref{solu1}) and then adapt the controller using \eqref{prop:K remixing}.
\end{enumerate}
Since the weight of the first virtual landmark is unrelated to the controller, it can always be calculated before the feedback controller design. Hence, the controller synthesis should consider the noise and covariance of the first virtual landmark (Prop.~\ref{prop:virtual landmark}) instead of  noise of the physical landmarks.

\section{Numerical Validation and Discussion} \label{section5}
In this section, we first demonstrate the application of Linear Multiple Estimation for multiple landmarks; then, we show the effectiveness the feedback control synthesis with a single virtual landmark against multiple physical landmarks. For all the plots, the blue (respectively, red) markers indicate physical (respectively, virtual) landmarks with randomly-generated positions; the corresponding ellipses represent the covariance of the measurements; solid straight lines indicates the boundaries of the environment.

\begin{figure}[htbp]
  \centering
  \includegraphics[width=\linewidth]{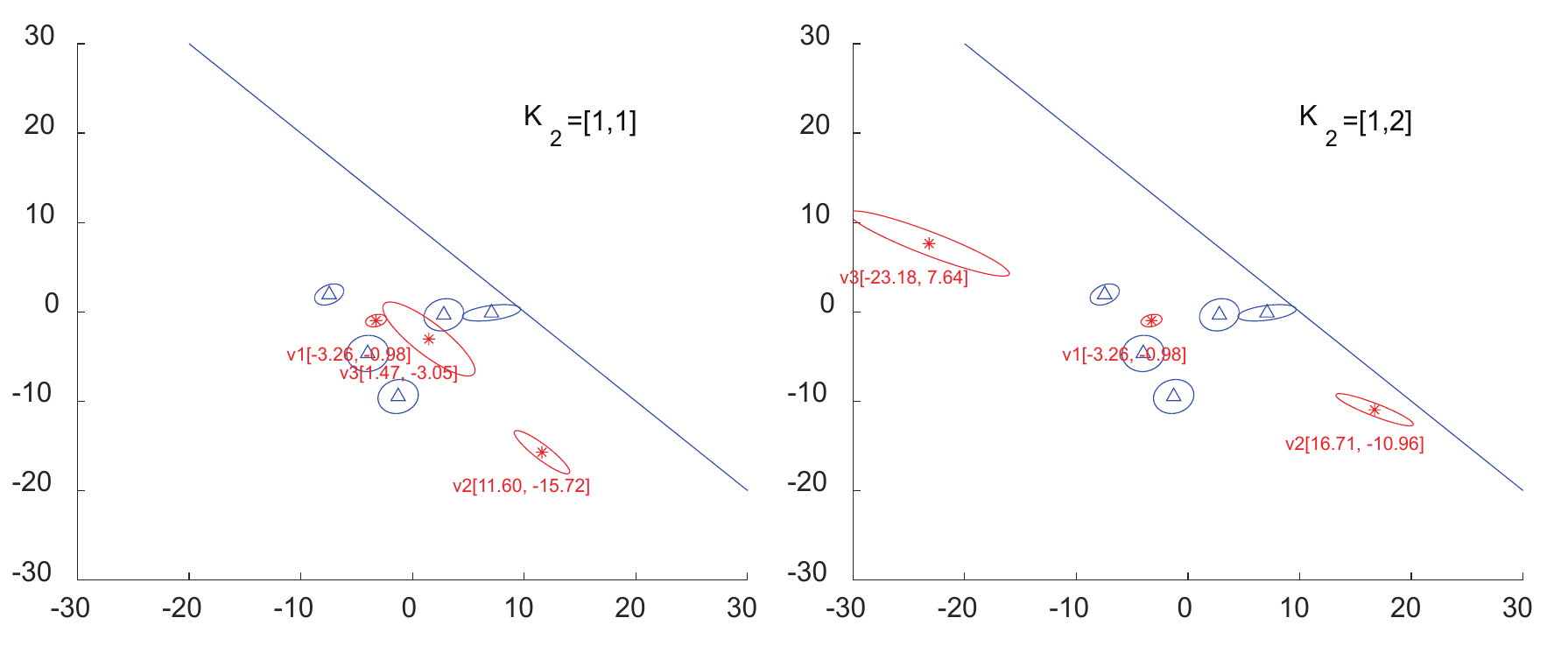}
  \caption{Five landmarks and corresponding top three virtual marks for a boundary. The coordinate of the three virtual marks are shown. The left case and the right case have different feedback matrices $K$ which leads to different $\bar{K}_2$.}
  \label{fig:LME2}
\end{figure}

\subsection{Virtual Landmarks for a Given Controller}
Fig.~\ref{fig:LME2} shows the result of using LME to obtain the top three virtual landmarks for two different controllers $K$ (an the corresponding $\bar{K}_2$). As expected, the first landmark is the same between the two cases, and has the smallest covariance.
The remaining virtual landmarks' covariance is minimized in a specific direction which is determined by the direction of the boundary and the controller together. 

\subsection{Effectiveness of Virtual-Landmark-Based Control}

\begin{figure}[htbp]
  \centering
  \includegraphics[width=\linewidth,trim=0mm 0mm 0mm 10mm, clip]{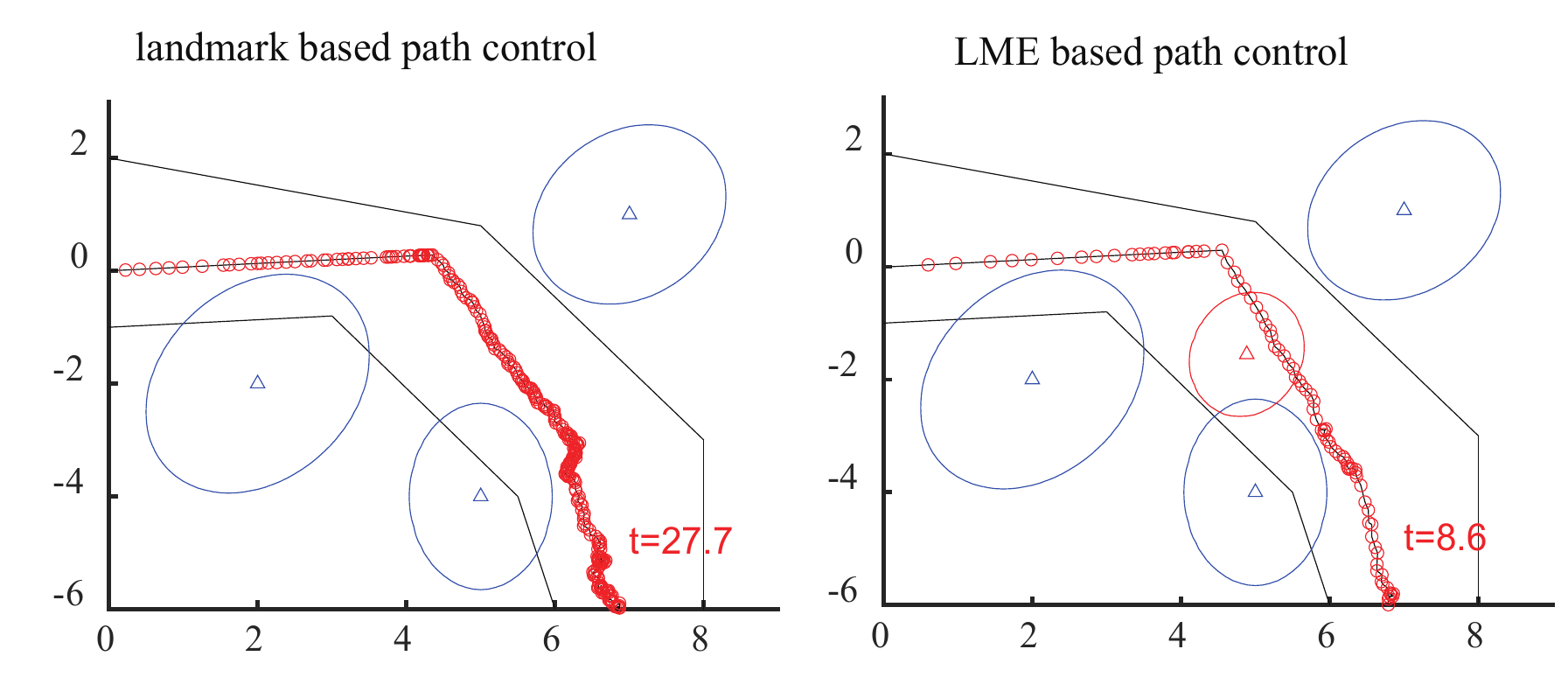}
  \caption{Simulation results with control synthesis based on all physical landmarks (left pane) and one virtual landmark (right pane). The robot starts at the origin $[0,0]$. The time to reach the exit face is indicated in red.}
  \label{fig:LMEvsLandmark}

\end{figure}

Fig.~\ref{fig:LMEvsLandmark} compares control synthesis with physical landmarks and a single virtual landmark ($l=1$) in a corridor-like environment. The simulation is based on a robot with single-integrator dynamics $\dot{x}=u$ and uses an Euler–Maruyama scheme with time step \unit[$0.1$]{s}. The robot inputs to the robot are represented by the relative displacement to the landmarks; for the sake of simplicity, we consider constant covariance matrices for the measurement noise  (i.e., which does not depend on time and position); note that, in this simulation, the covariance is large enough to be comparable to the width of the corridor. At each time step, the input $u$ is calculated using the controller $K$ computed from \eqref{eq:K opt dual}.

The controller that uses physical landmarks successfully avoids collisions with the walls, but the trajectory exhibits some chatter, and the control actions are more conservative, leading to a total time of \unit[$27.7$]{s} to exit the corridor.
The controller that uses the first virtual landmark (whose covariance is around one-third of each one of the physical landmarks) results in a trajectory with less jitter, with more aggressive control actions, reducing the total travel time to \unit[$8.6$]{s}.

\section{CONCLUSIONS} \label{section6}
In this paper, we propose the Linear Multiple Estimation (LME) method for landmark-based path planning and control, which generates uncorrelated virtual landmarks based on physical landmarks. 
The analytical results show that the first virtual landmark computed with LME is a weighted average of the physical landmarks with weights that are independent of the controller (which can then be designed at the same time) and the environment. Successive (statistically independent) landmarks instead depend on the controller used and the environment; nonetheless, the LME formulation provides a way to compute optimal weights for a given controller.

The numerical experiments shows that, even when a single landmark is used, the proposed strategy significantly reduces the effect of noise and improves the convergence speed.

Avenues for future work include a rigorous experimental evaluation on a real robot, using Kalman filters for the measurements and their time-variant covariance estimates, using LME to adjust the weights in real time.

\bibliographystyle{biblio/ieee}

\bibliography{bib/acc2021,bib/acc2022,%
  biblio/OtherFull,biblio/IEEEFull,biblio/IEEEConfFull,%
  biblio/tron,biblio/controlCBFs,biblio/controlCLF,biblio/planning,%
  biblio/slam%
}

\appendix
\subsection{Derivation of~\eqref{eq:K opt dual}}\label{sec:appendix K opt}
Eq.~\eqref{eq:K opt} can be rewritten by transforming the constraint $\forall x\in\cX$ into a worst-case analysis (the process is similar to the one used in~\cite{Bahreinian:ACC21}), leading to:
\begin{equation}
  \begin{aligned}
    \min_\vK&  \norm{\vK}^2_2\\
    \subjectto & \left[\begin{aligned}
        \max_{x} & -\bar{K}_1x\\
        \subjectto & A_xx+b_x\leq 0,
      \end{aligned} \right]\leq -\bar{K}_2\Gamma \bar{K}_2\transpose+\bar{k}_3
  \end{aligned}
\end{equation}

Taking the dual of the inner optimization problem, we obtain
\begin{equation}
  \begin{aligned}
    \min_\vK&  \norm{\vK}^2_2\\
    \subjectto & \left[\begin{aligned}
        \min_{\lambda} & b_x\transpose x\\
        \subjectto &  A_x\transpose\lambda=-\bar{K}_1 ,\\
        &\lambda\geq 0
      \end{aligned} \right]\leq -\bar{K}_2\Gamma \bar{K}_2\transpose+\bar{k}_3
  \end{aligned}
\end{equation}

Combining the minimization problems and rearranging we obtain~\eqref{eq:K opt dual} as desired.
\subsection{Proof of Prop.~\ref{prop:weights first}}\label{sec:appendix weights}

The Lagrangian of the constrained optimization \eqref{op1} is
\begin{multline}
    L=\bar{K}_2\sum_i{\left( W_i\Sigma _iW_{i}\transpose  \right)}\bar{K}_2\transpose +tr\left( \Lambda\transpose\left( \sum_i{W_i}-I \right) \right)\\
    =\sum_i{\bar{K}_2W_i\Sigma _iW_{i}\transpose\bar{K}_2\transpose +tr\left( \Lambda\transpose W_i \right)}-tr\left( \Lambda\transpose\right) .
\end{multline}
Imagining for the moment that $W$ is a function of time, the total derivative of $L$ is
\begin{multline}
  \dot{L}=\sum_i{\bar{K}_2\dot{W}_i\Sigma _iW_{i}\transpose\bar{K}_2\transpose +\bar{K}_2W_i\Sigma _i\dot{W}_{i}\transpose\bar{K}_2\transpose +tr\left( \Lambda\transpose\dot{W}_i \right)}\\
  =\sum_i{\trace\left( \left( 2\Sigma _iW_{i}\transpose\bar{K}_2\transpose\bar{K}_2+\Lambda\transpose\right) \dot{W}_i \right)}
\end{multline}
Extracting the gradient as the coefficient of $\dot{W}$, taking its transpose, and setting it to zero we have
\begin{equation}
  0=\left(\frac{\partial L}{\partial W_i}\right)\transpose\!\!\!\!=\left(2\Sigma _iW_{i}\transpose\bar{K}_2\transpose\bar{K}_2+\Lambda\transpose\right)\transpose
  \!\!=2\bar{K}_2\transpose\bar{K}_2W_i\Sigma _i+\Lambda,
  \end{equation}
  from which we can obtain an expression for $W_i$ as a function of the multiplier $\Lambda$
  \begin{equation}
    W_i=-\frac{1}{2}\left( \bar{K}_2\transpose\bar{K}_2 \right) ^{\dagger}\Lambda \Sigma _{i}^{-1}.
\end{equation}
If $ \bar{K}_2\transpose\bar{K}_2$ is singular, we need to add a penalty $\varepsilon \sum_i{W_iW_{i}\transpose }$ into the objective function. Then, we have
\begin{equation} \notag
  W_i=-\frac{1}{2}\left( \bar{K}_2\transpose\bar{K}_2 + \varepsilon I \right) ^{-1 }\Lambda \Sigma _{i}^{-1}.
\end{equation}

Substituting back in the constraint
\begin{equation} \notag
    I=\sum_i{W_i}= -\frac{1}{2}\left( \bar{K}_2\transpose\bar{K}_2+\varepsilon I \right) ^{-1}\Lambda \sum_i{\Sigma _{i}^{-1}}
  \end{equation}
  we can obtain a solution for $\Lambda$
\begin{equation}
    \Lambda =-2\left( \bar{K}_2\transpose\bar{K}_2+\varepsilon I \right) \left( \sum_i{\Sigma _{i}^{-1}} \right) ^{-1}
\end{equation}

Thus, the optimal weight $\hat{W}_i$ is give by
\begin{equation} \notag
  \hat{W}_i=\left( \sum_i{\Sigma _{i}^{-1}} \right) ^{-1}\Sigma _{i}^{-1}.
\end{equation}
Note that this result is consistent with standard results from Weighted Least Squares.

\section{Solution of the Optimization Problem \eqref{op2}}\label{sec:appendix LME}
Let $P_{i}^{\left( 1 \right)}=\Sigma _iW_{i}^{1}, ..., P_{i}^{\left( k-1 \right)}=\Sigma _iW_{i}^{k-1}$.
The Lagrangian function of~\eqref{op2} is given by
\begin{equation} \notag
  \begin{aligned}
    L&=\bar{K}_2 \sum_i{W_{i}^{k}\Sigma _iW_{i}^{kT}}\bar{K}_2\transpose+tr\left( \Lambda\transpose\left( \sum_i{W_{i}^{k}}-I \right) \right) \\
    &+tr\left( \varGamma _{1}\transpose\sum_i{P_{i}^{\left( 1 \right) T}W_{i}^{k}} \right) +\cdots \\
    &+tr\left( \varGamma _{k-1}\transpose\sum_i{P_{i}^{\left( k-1 \right) T}W_{i}^{k}} \right) \\
    \dot{L}&=\sum_i \bar{K}_2\dot{W}_{i}^{k}\Sigma _iW_{i}^{kT}\bar{K}_2\transpose +W_{i}^{k}\Sigma _i\dot{W}_{i}^{kT}+tr\left( \Lambda\transpose\dot{W}_{i}^{k} \right) \\
    &+tr\left( \varGamma _{1}\transpose P_{i}^{\left( 1 \right) T}\dot{W}_{i}^{k} \right) +\cdots +tr\left( \varGamma _{k-1}\transpose P_{i}^{\left( k-1 \right) T}\dot{W}_{i}^{k} \right) \\
    &=\sum_i{t}r((2\Sigma _iW_{i}^{kT}\bar{K}_2\transpose\bar{K}_2+\Lambda\transpose+\varGamma _{1}\transpose P_{i}^{(1)T}+\cdots
    \\
    & +\varGamma _{k-1}\transpose P_{i}^{(k-1)T})\dot{W}_{i}^{k})
  \end{aligned}
\end{equation}

Thus, the partial derivatives are
\begin{equation} \notag
  \begin{aligned}
    &\frac{\partial L}{\partial W_{i}^{k}}=2\Sigma _iW_{i}^{kT}\bar{K}_2\transpose\bar{K}_2+\Lambda\transpose+\varGamma _{1}\transpose P_{i}^{\left( 1 \right) T}+\cdots \\
    &\qquad\qquad\qquad\qquad \qquad\qquad+\varGamma _{k-1}\transpose P_{i}^{\left( k-1 \right) T}=0\\
    &2\bar{K}_2\transpose\bar{K}_2W_{i}^{k}\Sigma _i+\Lambda +P_{i}^{\left( 1 \right)}\varGamma _1+\cdots \\
    &\qquad\qquad\qquad\qquad\qquad\qquad +P_{i}^{\left( k-1 \right)}\varGamma _{k-1}=0\\
    &W_{i}^{k}=-\frac{1}{2}(\bar{K}_2\transpose\bar{K}_2)^{-1}(\Lambda +P_{i}^{(1)}\varGamma _1+\cdots\\
    &\qquad\qquad\qquad\qquad\qquad\qquad +P_{i}^{(k-1)}\varGamma _{k-1})\Sigma _{i}^{-1}
  \end{aligned}
\end{equation}

To find the parameter, we have \eqref{apx_f1}.

\begin{figure*}[htbp]
  \centering
\begin{equation} \label{apx_f1}
  \begin{aligned}
    \sum_i{W_{i}^{k}}=I\Longrightarrow \left( K_2 ^TK_2 \right) ^{-1}\sum_i{\left( \Lambda +P_{i}^{\left( 1 \right)}\varGamma _1+\cdots +P_{i}^{\left( k-1 \right)}\varGamma _{k-1} \right) \Sigma _{i}^{-1}}=-2I
    \\
    \sum_i{P_{i}^{\left( 1 \right) T}W_{i}^{k}=0}\Longrightarrow \sum_i{P_{i}^{\left( 1 \right) T}\left( K_2 ^TK_2 \right) ^{-1}\left( \Lambda +P_{i}^{\left( 1 \right)}\varGamma _1+\cdots +P_{i}^{\left( k-1 \right)}\varGamma _{k-1} \right) \Sigma _{i}^{-1}}=0
    \\
    \vdots \qquad\qquad\qquad\qquad\qquad\qquad\qquad\qquad\qquad\qquad
    \\
    \sum_i{P_{i}^{\left( k-1 \right) T}W_{i}^{k}=0}\Longrightarrow \sum_i{P_{i}^{\left( k-1 \right) T}\left( K_2 ^TK_2 \right) ^{-1}\left( \Lambda +P_{i}^{\left( 1 \right)}\varGamma _1+\cdots +P_{i}^{\left( k-1 \right)}\varGamma _{k-1} \right) \Sigma _{i}^{-1}}=0
  \end{aligned}
\end{equation}
    \begin{equation}\label{eq:sums}
    A \defeq
    \left[ \begin{matrix}
	\sum_i{C_i\otimes H^{-1}}&		\sum_i{C_i\otimes \left( H^{-1}P_{i}^{\left( 1 \right)} \right)}&		\dots&		\sum_i{C_i\otimes \left( H^{-1}P_{i}^{\left( k-1 \right)} \right)}\\
	\sum_i{C_i\otimes \left( P_{i}^{\left( 1 \right) T}H^{-1} \right)}&		\sum_i{C_i\otimes \left( P_{i}^{\left( 1 \right) T}H^{-1}P_{i}^{\left( 1 \right)} \right)}&		\dots&		\sum_i{C_i\otimes \left( P_{i}^{\left( 1 \right) T}H^{-1}P_{i}^{\left( k-1 \right)} \right)}\\
	\vdots&		\vdots&		\vdots&		\vdots\\
	\sum_i{C_i\otimes \left( P_{i}^{\left( k-1 \right) T}H^{-1} \right)}&		\sum_i{C_i\otimes \left( P_{i}^{\left( k-1 \right) T}H^{-1}P_{i}^{\left( 1 \right)} \right)}&		\dots&		\sum_i{C_i\otimes \left( P_{i}^{\left( k-1 \right) T}H^{-1}P_{i}^{\left( k-1 \right)} \right)}\\
\end{matrix} \right]
\end{equation}
\rule{0.9\textwidth}{0.71pt}
\end{figure*}


Let $H=\bar{K}_2\transpose\bar{K}_2$, $C_i=\Sigma _{i}^{-1}$, we get
\begin{equation} \notag
  Ax=b
\end{equation}

\begin{equation*}
  A \defeq \eqref{eq:sums} \quad
  x\defeq \left[ \begin{array}{c}
                   \mathrm{vec}\left( \Lambda \right)\\
                   \mathrm{vec}\left( \varGamma _1 \right)\\
                   \vdots\\
                   \mathrm{vec}\left( \varGamma _{k-1} \right)\\
                 \end{array} \right] \quad
               b\defeq \left[ \begin{array}{c}
                                -2\mathrm{vec}\left( I \right)\\
                                0\\
                                \vdots\\
                                0\\
                              \end{array} \right]
                          \end{equation*}

                          Solving the linear equations, we can obtain the parameter $\mathrm{vec}(\Lambda)$,$\mathrm{vec}(\varGamma_1)$,...,$\mathrm{vec}(\varGamma_k-1)$.

                          Then, we have
                          \begin{equation} \label{apxf2}
                            \begin{aligned}
                              W_{i}^{k}&=-\frac{1}{2}\left( \bar{K}_2\transpose\bar{K}_2 \right) ^{-1}(\Lambda +P_{i}^{\left( 1 \right)}\varGamma _1\\
                              &\qquad\qquad\qquad +\cdots +P_{i}^{\left( k-1 \right)}\varGamma _{k-1})\Sigma _{i}^{-1}
                            \end{aligned}
                          \end{equation}

                          If $\bar{K}_2\transpose  \bar{K}_2 $ is singular, we need to add a small penalty $\varepsilon \sum_i{W_iW_{i}\transpose }$ into the objective function to avoid singular solutions. Then, the result is

                          \begin{equation} \label{apxf3}
                            \begin{aligned}
                              W_{i}^{k}=&-\frac{1}{2}\left( \bar{K}_2\transpose\bar{K}_2+\epsilon I \right) ^{-1}(\Lambda +P_{i}^{\left( 1 \right)}\varGamma _1\\
                              &\qquad\qquad\qquad +\cdots +P_{i}^{\left( k-1 \right)}\varGamma _{k-1})\Sigma _{i}^{-1},
                            \end{aligned}
                          \end{equation}
                          where $\mathrm{vec}(\Lambda),\mathrm{vec}(\varGamma_1),\cdots,\mathrm{vec}(\varGamma_{k-1})$ are also changed with $H=\bar{K}_2\transpose\bar{K}_2 +\epsilon I$.

\end{document}